\documentclass{article}
\pdfoutput=1

\usepackage{mlmodern}
\usepackage{microtype}
\usepackage{amsmath}
\usepackage{amssymb}
\usepackage{amsthm}
\usepackage{enumerate}
\usepackage{tikz-cd}

\usepackage{hyperref}
\usepackage[numbers]{natbib}
\usepackage{cleveref}

\renewcommand{\phi}{\varphi}

\theoremstyle{plain}
\newtheorem{thm}{Theorem}
\crefname{thm}{Thm.}{Thms.}
\newtheorem{prop}{Proposition}
\crefname{prop}{Prop.}{Props.}

\theoremstyle{definition}

\theoremstyle{remark}

\title{Hausdorff reflections and bifurcate curves}
\author{John Dougherty}

\begin{document}

\maketitle

\begin{abstract}
  A manifold is a space that locally looks like the smooth space
  $\mathbf{R}^{n}$.  It is usually also assumed that the underlying topological
  space of a manifold is hausdorff.  However, there are natural examples of
  manifolds for which the hausdorff conditions fails.  Some but not all of these
  examples contain bifurcate pairs of curves: pairs of curves that agree on some
  initial interval but disagree on a later interval.  The first part of this
  note proves that a manifold $M$ is hausdorff if and only if (i) it contains no
  bifurcate curves and (ii) there is a hausdorff manifold $N$ with the same
  algebra of smooth real-valued functions as $M$; this confirms a conjecture of
  Wu and Weatherall.  The second part of this note shows that a hausdorff
  manifold $N$ satisfying (ii) is a certain quotient of $M$.
\end{abstract}

\section{Introduction}

A manifold is a topological space equipped with a structure making it locally
look like the smooth cartesian space $\mathbf{R}^{n}$.  It is usually also
assumed that the topological space is hausdorff and second countable, and these
topological conditions are necessary for structural results like embedding and
classification theorems \citep{Kolar1993,Milnor1965,Milnor1974,Whitney1936}.
However, much of the theory of smooth manifolds can be developed without these
topological conditions \citep{Bourbaki1971,Wedhorn2016}.  And there are natural
examples of locally cartesian spaces for which these topological conditions
fail, like \'etale spaces of sheaves, leaf spaces of foliations, and certain
spaces appearing in spacetime physics \citep{Haefliger1957,Hajicek1971b}.  This
note is concerned with the situation in which the hausdorff condition is not
assumed.  We therefore say ``manifold'' to mean a second countable topological
space equipped with a smooth structure, ``hausdorff manifold'' to mean a
manifold whose underlying topological space is hausdorff, and ``non-hausdorff
manifold'' to mean a manifold whose underlying topological space is not
hausdorff.

Removing the hausdorff condition allows manifolds to exhibit two properties not
found among hausdorff manifolds.  First, non-hausdorff manifolds contain points
that cannot be distinguished by any smooth real-valued function.  A topological
space $M$ is hausdorff just in case the diagonal $\Delta_{M} \subseteq M \times
M$ is closed.  So if $M$ is a non-hausdorff manifold, it contains distinct
points $p$ and $q$ such that $(p, q)$ belongs to the closure of the diagonal.
Let $\alpha : M \to \mathbf{R}$ be a \emph{functional} on $M$: that is, a smooth
real-valued function.  Any functional is continuous, and $\mathbf{R}$ is
hausdorff, so the preimage $(\alpha \times \alpha)^{-1}(\Delta_{\mathbf{R}})$ of
the diagonal is closed and contains the diagonal $\Delta_{M}$, implying that $\alpha(p) =
\alpha(q)$.  

It follows that distinct manifolds can have the same algebra of functionals,
which cannot happen in the hausdorff case.  In particular, the
$\mathbf{R}$-algebra $C^{\infty}(M)$ of functionals on a non-hausdorff manifold
$M$ might coincide with the $\mathbf{R}$-algebra $C^{\infty}(N)$ of functionals
on a hausdorff manifold $N$.
For example, let
$\mathbf{R}_{\Bumpeq}$ be the line with two origins, obtained by gluing two
copies of the real line along the open subset $\mathbf{R} \setminus \{0\}$.  A
functional on $\mathbf{R}_{\Bumpeq}$ is then a pair of functionals $(\alpha_{1},
\alpha_{2})$ on $\mathbf{R}$ that agree on $\mathbf{R} \setminus \{0\}$.  So
$(\alpha_{1}, \alpha_{2})^{-1}(\Delta_{\mathbf{R}})$ is closed and contains
$\mathbf{R}
\setminus \{0\}$, implying that $\alpha_{1} = \alpha_{2}$ and 
identifying functionals on $\mathbf{R}_{\Bumpeq}$ with functionals on $\mathbf{R}$.  

More generally, a \emph{hausdorff reflection} for a manifold $M$ is a hausdorff
manifold $N$ along with an isomorphism of $\mathbf{R}$-algebras $\phi :
C^{\infty}(N) \to C^{\infty}(M)$.  Because the embedding of manifolds into
$\mathbf{R}$-algebras is full and faithful, the hausdorff reflection of a
manifold is unique up to unique isomorphism when it exists.  Any hausdorff
manifold is naturally its own hausdorff reflection, with $\phi$ the identity.
The line with two origins shows that some non-hausdorff manifolds have a
hausdorff reflection.  On the other hand, gluing two copies of $\mathbf{R}$
along $(-\infty, 0)$ gives a non-hausdorff manifold without a hausdorff
reflection.

The second novel feature of non-hausdorff manifolds is the possibility of
bifurcate curves.  A pair 
of smooth curves $\gamma_{1}, \gamma_{2} : [0, 1] \to M$ 
in a manifold $M$ is \emph{bifurcate} if there is some
$0 < s \leq 1$ such that $\gamma_{1}(t) = \gamma_{2}(t)$ for $t < s$ and
$\gamma_{1}(t) \not= \gamma_{2}(t)$ for $s \leq t$.  If $(\gamma_{1},
\gamma_{2})$ is a bifurcate pair of curves in a manifold $M$, then $(\gamma_{1},
\gamma_{2})^{-1}(\Delta_{M}) = [0, s)$ isn't closed, so $M$ isn't hausdorff.
Some non-hausdorff manifolds admit bifurcate curves: if $i_{1}, i_{2} :
\mathbf{R} \to \mathbf{R}_{\Bumpeq}$ are the inclusions of the two copies of
$\mathbf{R}$ into the line with two origins and $\gamma : [0, 1] \to \mathbf{R}$
is the curve $\gamma(t) = t-1$, then $(i_{1}\gamma, i_{2}\gamma)$ is a bifurcate pair of curves.  But not all non-hausdorff manifolds do, as evidenced by
Misner spacetime with two extensions \citep[p.~171--174]{Hawking1973}.  

The point of this note is to relate the hausdorff condition to hausdorff
reflections and bifurcate curves.  The first main result is the following:
\begin{thm}
  \label{thm:main}
  For any manifold $M$, the following are equivalent:
  \begin{enumerate}
    \item $M$ is hausdorff
    \item $M$ has a hausdorff reflection and lacks bifurcate curves.
  \end{enumerate}
\end{thm}
\noindent
We have already seen that the forward direction holds.  For the backward
direction, we show that if $M$ is a non-hausdorff manifold with a hausdorff
reflection $N$, then some curves in $N$ admit multiple lifts to smooth curves in
$M$ that agree on an initial interval.  This proves a conjecture of
Wu and Weatherall
\citep{Wu2023}, which inspired this note.

The equivalence in \cref{thm:main} is most interesting in cases where we know
whether $M$ satisfies two of the three conditions appearing in it.  One
generally knows whether some manifold of interest is hausdorff.  H\'aji\v{c}ek
\citep{Hajicek1971a} gives a criterion for gluings of hausdorff manifolds to
produce bifurcate curves and produces examples of non-hausdorff manifolds
without bifurcate curves. It follows from \cref{thm:main} that these manifolds
have no hausdorff reflections.  In the second part of this note we show that the
hausdorff reflection coincides with a certain quotient.
\begin{thm}
  \label{thm:quotient}
  Let $M$ be a manifold, and let $E \subseteq M \times M$ be the set of pairs
  $(p, q)$ such that $\alpha(p) = \alpha(q)$ for all functionals $\alpha$ on
  $M$.  If the quotient manifold $\eta : M \to M/E$ exists, then $\eta^{*} :
  C^{\infty}(M/E) \to C^{\infty}(M)$ is a hausdorff reflection for $M$.
  Conversely, every hausdorff reflection is of this form.
\end{thm}
\noindent
It follows that the hausdorff reflection exists if and only if $E$ is a
submanifold of $M \times M$ and the projection $E \to M$ onto the first factor
is a submersion \citep[\S5.9.5]{Bourbaki1971}.  The forward direction follows
from the universal property of the quotient.  To prove the converse, we move to
a context more general than manifolds in which the hausdorff reflection and
quotient always exist and coincide.

\section{The hausdorff reflection as a manifold}

A hausdorff reflection for a manifold $M$ is an isomorphism of
$\mathbf{R}$-algebras $\phi : C^{\infty}(N) \to C^{\infty}(M)$ with $N$
hausdorff.  This definition is somewhat inconvenient because it involves data in
two categories: manifolds and $\mathbf{R}$-algebras.  When $M$ and $N$ are both
hausdorff, the full and faithful embedding of hausdorff manifolds in
$\mathbf{R}$-algebras implies that an isomorphism of $\mathbf{R}$-algebras
$C^{\infty}(N) \to C^{\infty}(M)$ is the same thing as a diffeomorphism $M \to
N$, resolving this inconvenience.  But when we drop the hausdorff assumption,
this no longer holds; this is what allowed for nontrivial hausdorff reflections
in the first place.  However, as this section shows, because $N$ is hausdorff we
can recover enough manifold data about the reflection to prove \cref{thm:main}.
  
For any manifold $M$, the set $C^{\infty}(M)$ of functionals on $M$ is an
$\mathbf{R}$-algebra when endowed with pointwise addition, multiplication, and
multiplication by scalars.  And pullback along any smooth function $f : M \to N$
gives an $\mathbf{R}$-algebra homomorphism $f^{*} : C^{\infty}(N) \to
C^{\infty}(M)$.  Since pullback commutes with composition, this defines a
functor
\[
  C^{\infty} :
  \mathrm{Man}^{\text{op}}
  \to
  \mathrm{Alg}_{\mathbf{R}}
\]
from the category of manifolds and smooth functions to the category of
$\mathbf{R}$-algebras and $\mathbf{R}$-algebra homomorphisms.  

When restricted to the full subcategory $\mathrm{HMan}$ of $\mathrm{Man}$ on the
hausdorff manifolds, this functor has two salient properties.  First, if $f : M
\to N$ is a function on the underlying sets of hausdorff manifolds $M$ and $N$
such that $f^{*} : C^{\infty}(N) \to C^{\infty}(M)$ is an $\mathbf{R}$-algebra
homomorphism, then $f : M \to N$ is smooth.  Second, for any hausdorff manifold
$N$ the map
\[
  \operatorname{ev} : N \to \mathrm{Alg}_{\mathbf{R}}(C^{\infty}(N), \mathbf{R})
  \qquad\qquad
  \operatorname{ev}_{q}(\alpha)
  =
  \alpha(q)
\]
is a bijection, a result sometimes called ``Milnor's exercise''
\citep[\S35.9]{Kolar1993}.  It follows from these two properties that the
restriction of the functor $C^{\infty}$ to $\mathrm{HMan}$ is full and faithful.  

Neither of these two conditions is true for non-hausdorff manifolds.  For
example, consider again the inclusions $i_{1}, i_{2} : \mathbf{R} \to
\mathbf{R}_{\Bumpeq}$ of the two copies of $\mathbf{R}$ into the line with two
origins.  The curve $\gamma : \mathbf{R} \to \mathbf{R}_{\Bumpeq}$ satisfying
$\gamma(t) = i_{1}(0)$ for $t \leq 0$ and $\gamma(t) = i_{2}(0)$ otherwise
composes with any functional to give a smooth constant map, so $\gamma^{*} :
C^{\infty}(\mathbf{R}_{\Bumpeq}) \to C^{\infty}(\mathbf{R})$ is an
$\mathbf{R}$-algebra homomorphism, but $\gamma$ isn't smooth; indeed, it's not
even continuous, because the preimage of any cartesian open neighborhood $U$ of
$i_{1}(0)$ is $\gamma^{-1}(U) = (-\infty, 0]$.  And for any functional $\alpha$
on $\mathbf{R}_{\Bumpeq}$ we have $i_{1}^{*}\alpha = i_{2}^{*}\alpha$ by the
argument in the introduction, meaning that $C^{\infty}$ isn't faithful.  

However, examining the proofs of these results shows that they still go through
when the codomain is hausdorff.
\begin{prop}
  \label{prop:Coo_reflect_smoothness}
  For any manifolds $M$ and $N$, if $N$ is hausdorff then the square
  \[
    \begin{tikzcd}
      {\mathrm{Man}(M, N)} \ar[r, hook] \ar[d ] 
      \ar[dr, phantom, "\lrcorner" very near start]
      &
      {\mathrm{Set}(M, N)} \ar[d] \\
      {\mathrm{Alg}_{\mathbf{R}}(C^{\infty}(N), C^{\infty}(M))} \ar[r, hook] &
      {\mathrm{Set}(
        C^{\infty}(N),
        \mathrm{Set}(M, \mathbf{R})
      )}
    \end{tikzcd}
  \]
  is a pullback.
\end{prop}
\begin{proof}
  Since the top arrow is an injection, it suffices to show that any set function
  $f : M \to N$ that induces an $\mathbf{R}$-algebra homomorphism $f^{*} :
  C^{\infty}(N) \to C^{\infty}(M)$ is smooth.  Let $f$ be such a function.

  For any an open subset $V$ of $N$ and any $p$ in $f^{-1}(V)$, the fact that
  $N$ is hausdorff means we can use the usual partition of unity argument to
  construct a functional $\alpha : N \to \mathbf{R}$ that vanishes outside of
  $V$ and satisfies $\alpha(f(p)) = 1$.  Since $f^{*}\alpha$ is a functional on
  $M$, the set $(f^{*}\alpha)^{-1}(\mathbf{R} \setminus \{0\})$ is an open
  neighborhood of $p$ contained in $f^{-1}(V)$.  Therefore $f$ is continuous.

  For any point $p$ of $M$, choose a chart $(V, y)$ around $f(p)$.  For any
  coordinate $y^{i}$ on $V$, the fact that $N$ is hausdorff means we can use the
  usual partition of unity argument to construct a functional $\alpha$ on $N$
  that agrees with $y^{i}$ on an open neighborhood $V_{0} \subseteq V$ of
  $f(p)$.  Then $f^{*}\alpha$ is smooth by hypothesis and agrees with $y^{i}
  \cdot f$ on $f^{-1}(V_{0})$.  Since $f$ is continuous, this means that $y^{i}
  \cdot f$ is smooth on an open neighborhood of $p$ and thus that $f$ is smooth
  at $p$.
\end{proof}

\begin{prop}
  \label{prop:Coo_embedding}
  For any manifolds $M$ and $N$, if $N$ is hausdorff then the map
  \[
    \mathrm{Man}(M, N)
    \to
    \mathrm{Alg}_{\mathbf{R}}(C^{\infty}(N), C^{\infty}(M))
  \]
  is a bijection.
\end{prop}

\begin{proof}
  If $f, g : M \to N$ are distinct set functions, then there's some $p$ in $M$
  at which $f$ and $g$ disagree.  Since $N$ is hausdorff, we can use the usual
  partition of unity argument to construct a functional $\alpha$ on $N$ such
  that $\alpha(f(p)) = 1$ and $\alpha(g(p)) = 0$.  So $f^{*}\alpha$ and
  $g^{*}\alpha$ are distinct, making the right leg of the pullback in
  \cref{prop:Coo_reflect_smoothness} an injection.  Since pullbacks preserve
  monos, the left leg is an injection as well.

  For surjectivity, consider any
  $\mathbf{R}$-algebra homomorphism $\phi : C^{\infty}(N) \to C^{\infty}(M)$.
  Since $N$ is hausdorff, the full and faithful embedding of hausdorff manifolds
  into $\mathbf{R}$-algebras gives a map
  \[
    f
    :
    M
    \to
    \mathrm{Alg}_{\mathbf{R}}(C^{\infty}(N), \mathbf{R})
    \cong
    N
    \qquad\qquad
    f(p) = \mathrm{ev}_{p} \cdot \phi
  \]
  For any functional $\alpha$ on $N$ and point $p$ of $M$ we then have
  $\alpha(f(p)) = \phi(\alpha)(p)$, from which it follows that $f^{*} = \phi$,
  making $f$ smooth by \cref{prop:Coo_reflect_smoothness} and the map in the
  statement a bijection.
\end{proof}

In light of these results, a hausdorff reflection for a manifold $M$ is a smooth
function $\eta : M \to N$ such that $\eta^{*} : C^{\infty}(N) \to C^{\infty}(M)$
is an isomorphism of $\mathbf{R}$-algebras.  When $M$ is hausdorff, the fact
that $C^{\infty}$ is full and faithful on hausdorff manifolds implies that
$\eta$ is a diffeomorphism.  When $M$ is non-hausdorff, the argument from the
introduction shows that $\eta$ must identify any pair belonging to the closure
of the diagonal, since any map into a hausdorff space will, and so
$\eta$ isn't injective.  But it is still a diffeomorphism locally, and this
suffices to prove \cref{thm:main}:
\begin{prop}
  \label{prop:local_diffeo}
  Let $\eta : M \to N$ be a smooth function of manifolds such that $\eta^{*} :
  C^{\infty}(N) \to C^{\infty}(M)$ is an isomorphism of $\mathbf{R}$-algebras.
  Each point $p$ of $M$ has an open neighborhood $U$ such that
  $\eta\rvert_{U} : U \to \eta(U)$ is a diffeomorphism.
\end{prop}
\begin{proof}
  At any point $p$, the tangent space $T_{p}M$ is the vector space of
  derivations on $C^{\infty}(M)$ at $p$, so the isomorphism $\eta^{*} :
  C^{\infty}(N) \to C^{\infty}(M)$ induces an isomorphism $\eta_{*} : T_{p}M \to
  T_{\eta(p)}N$.  Choosing a chart $V$ around $\eta(p)$ and a chart $U$ around
  $p$ with $U \subseteq \eta^{-1}(V)$, the inverse function theorem applied to
  $(\eta\rvert_{U})_{*} : T_{p}U \to T_{\eta(p)}V$ supplies an open neighborhood
  $U_{0} \subseteq U$ of $p$ such that $\eta\rvert_{U_{0}} : U_{0} \to
  \eta(U_{0})$ is a diffeomorphism.
\end{proof}

\begin{proof}[Proof (\cref{thm:main})]
  Suppose that $M$ is a non-hausdorff manifold and $\eta : M \to N$ its
  hausdorff reflection.  Since $M$ isn't hausdorff, it contains distinct points
  $p_{1}$ and $p_{2}$ such that $(p_{1}, p_{2})$ belongs to the closure of the
  diagonal in $M \times M$.  Since $\eta$ is a local diffeomorphism, we can
  choose cartesian open neighborhoods $U_{1}$ of $p_{1}$ and $U_{2}$ of $p_{2}$
  and a cartesian open neighborhood $V$ of $\eta(p_{1}) = \eta(p_{2})$ such that
  the restrictions $\eta_{1} : U_{1} \to V$ and $\eta_{2} : U_{2}
  \to V$ of $\eta$ are diffeomorphisms.  On $U = U_{1} \cap U_{2}$ the
  restrictions of $\eta_{1}$ and $\eta_{2}$ are both the restriction
  $\eta\rvert_{U}$, so $\eta_{1}^{-1}$ and $\eta_{2}^{-1}$ coincide on
  $\eta(U)$; elsewhere they differ.

  Now $U_{1} \times U_{2}$ is an open neighborhood of $(p_{1}, p_{2})$ in $M
  \times M$, and since $(p_{1}, p_{2})$ belongs to the closure of the diagonal
  it follows that $U_{1} \times U_{2}$ meets the diagonal, giving some point $q$
  in $U$.  And since $U_{1}$ is hausdorff and contains $p_{1}$, it can't contain
  $p_{2}$.  Therefore $\eta(U)$ is an inhabited proper open subset of the
  cartesian space $V$.  So we can choose a smooth curve $\gamma : [0, 1] \to V$
  such that $\gamma(t)$ belongs to $\eta(U)$ for all $t < 1$ and $\gamma(1)$ is
  on the boundary of $\eta(U)$.  From this we obtain smooth curves $\gamma_{1} =
  \eta_{1}^{-1} \cdot \gamma$ and $\gamma_{2} = \eta_{2}^{-1} \cdot \gamma$ in
  $M$.  Since $\eta_{1}^{-1}$ and $\eta_{2}^{-1}$ agree on $\eta(U)$ we have
  $\gamma_{1}(t) = \gamma_{2}(t)$ for $t < 1$, and since $\eta_{1}^{-1}$ and
  $\eta_{2}^{-1}$ differ on the complement of $\eta(U)$ we have $\gamma_{1}(1)
  \not= \gamma_{2}(1)$.  So $(\gamma_{1}, \gamma_{2})$ is a bifurcate pair of
  curves in $M$.  
\end{proof}

\section{The hausdorff reflection as a quotient}

The results of the previous section suggest that we can think of a hausdorff
reflection $\eta : M \to N$ as the quotient identifying those points that cannot
be distinguished by functionals.  When the quotient exists, a short argument
shows that this is indeed the case.  If we don't assume that the quotient
exists, it's still possible to characterize the hausdorff reflection as a
colimit determined by $M$.

One direction of \cref{thm:quotient} follows from the universal property of the
quotient:
\begin{prop}[\cref{thm:quotient}, forward]
  If the quotient $\eta : M \to M/E$ exists, then it is a hausdorff reflection
  for $M$.
\end{prop}
\begin{proof}
  Since $E$ is the intersection of closed sets of the form $(\alpha \times
  \alpha)^{-1}(\Delta_{\mathbf{R}})$ it's closed, making $M/E$ hausdorff
  \citep[\S5.9.5]{Bourbaki1971}.
  The universal property of the quotient says that the $\mathbf{R}$-algebra
  homomorphism $\eta^{*} : C^{\infty}(M/E) \to C^{\infty}(M)$ is an injection
  whose image is the set of functionals $\alpha$ such that $\alpha(p) =
  \alpha(q)$ for all $(p, q)$ in $E$.  Since this is all functionals, the map
  $\eta^{*}$ a surjection.
\end{proof}

This leaves open the possibility that the hausdorff reflection may exist in
cases where the quotient $M/E$ does not.  The bijection of
\cref{prop:Coo_embedding} gives a first step toward closing off this
possibility:
\begin{prop}
  \label{cor:refl_colim}
  Let $\eta : M \to N$ be a hausdorff reflection, let $M_{\bullet}$ be a diagram
  in the category of hausdorff manifolds, and let $\theta :
  M_{\bullet} \to M$ be a colimiting cocone in the category of manifolds.  Then
  $\eta \cdot \theta : M_{\bullet} \to N$ is colimiting in the
  full subcategory of hausdorff manifolds.
\end{prop}
\begin{proof}
  Since $C^{\infty}$ is full and faithful on the subcategory of hausdorff
  manifolds it reflects limits, so it suffices to show that
  $\theta^{*} \cdot \eta^{*} : C^{\infty}(N) \to
  C^{\infty}(M_{\bullet})$ is a limiting cone of $\mathbf{R}$-algebras.  Since
  $\eta^{*}$ is an isomorphism, this is the same as showing that
  $\theta^{*} : C^{\infty}(M) \to C^{\infty}(M_{\bullet})$ is a
  limiting cone of $\mathbf{R}$-algebras, for which it suffices to show that
  $C^{\infty}$ preserves limits on all of $\mathrm{Man}$.

  Consider the composite
  \[
    \mathrm{Man}^{\text{op}}
    \xrightarrow{C^{\infty}} \mathrm{Alg}_{\mathbf{R}}
    \xrightarrow{U} \mathrm{Set}
  \]
  with $U$ the underlying set functor.  This composite preserves limits because
  it's represented by the manifold $\mathbf{R}$.  Since $U$ is monadic it
  reflects limits, from which it follows that $C^{\infty}$ preserves limits.
\end{proof}

Every manifold is a colimit of hausdorff spaces, more or less by definition, and
so any hausdorff reflection that exists is a colimit in the subcategory of
hausdorff manifolds.  More concretely, any manifold $M$ is second countable,
hence admits a countable cover $\{U_{i}\}_{i \in I}$ by open charts.  Since
smoothness is a local property, we have a coequalizer
\[
  \begin{tikzcd}
    {\coprod_{i, j \in I} U_{i} \cap U_{j}}
    \ar[r, shift left] \ar[r, shift right] &
    {\coprod_{i \in I} U_{i}} \ar[r] &
    M
  \end{tikzcd}
\]
where the parallel arrows are the inclusions of $U_{i} \cap U_{j}$ into $U_{i}$
and $U_{j}$, respectively.  Since the $U_{i}$ are all charts, hence hausdorff,
the hausdorff reflection of $M$ is the coequalizer of this diagram in the
subcategory of hausdorff manifolds.

\section{The hausdorff reflection as a fr\"olicher space}

\cref{cor:refl_colim} is merely a first step because colimits of manifolds
don't exist in general and aren't systematically computable when they do.  This
is naturally solved by embedding manifolds in a larger category where
colimits exist.  Since we are concerned with topological matters, this larger
category should be small enough to admit a reasonable theory of topology.
One option suitable for these purposes is the category of fr\"olicher spaces
\citep{Froelicher1980,Froelicher1982,Kriegl1997,nLab,Stacey2011}.  
These are particularly convenient because hausdorff fr\"olicher spaces form an
honest reflective subcategory of all fr\"olicher spaces, so we can always
compute the hausdorff reflection as a fr\"olicher space.  This will coincide
with the hausdorff reflection as a manifold when the latter exists.

A \emph{fr\"olicher space} is a triple $(X, C_{X}, F_{X})$ consisting of a set
$X$, a set $C_{X}$ of set functions $\mathbf{R} \to X$, and a set $F_{X}$ of set
functions $X \to \mathbf{R}$ such that
\begin{enumerate}[(i)]
  \item a set function $\alpha : X \to \mathbf{R}$ belongs to $F_{X}$ if and
    only if $\alpha \cdot \gamma : \mathbf{R} \to \mathbf{R}$ is smooth for all
    $\gamma$ in $C_{X}$, and 
  \item a set function $\gamma : \mathbf{R} \to X$ belongs to $C_{X}$ if and
    only if $\alpha \cdot \gamma : \mathbf{R} \to \mathbf{R}$ is smooth for all
    $\alpha$ in $F_{X}$.
\end{enumerate}
We call $C_{X}$ the set of \emph{curves} in $X$ and $F_{X}$ the set of
\emph{functionals} on $X$, and we refer to a fr\"olicher space by its underlying
set.  A map $f : X \to Y$ of fr\"olicher spaces is a function of the underlying
sets satisfying the following equivalent conditions
\begin{enumerate}[(i)]
  \item for every curve $\gamma$ in $C_{X}$, the composite $f \cdot \gamma$ is a
    curve in $C_{Y}$;
  \item for every functional $\alpha$ in $F_{Y}$, the composite $\alpha \cdot f$
    is a functional in $F_{X}$; and
  \item for every curve $\gamma$ in $C_{X}$ and functional $\alpha$ in $F_{Y}$,
    the composite $\alpha \cdot f \cdot \gamma$ is smooth.
\end{enumerate}
\begin{prop}[\citep{Froelicher1980}]
  The category $\mathrm{Froe}$ of fr\"olicher spaces is complete, cocomplete,
  and cartesian closed.  The underlying set functor is topological, hence a
  faithful, amnestic isofibration.
\end{prop}

The category of manifolds is a subcategory of the category of fr\"olicher
spaces, with the subcategory of hausdorff manifolds a full subcategory
\begin{prop}
  Any manifold $M$ gives a fr\"olicher space $FM = (M, C_{M}, F_{M})$, where
  $F_{M} = \mathrm{Man}(M, \mathbf{R})$.  Any smooth function $f : M \to N$
  gives a map of fr\"olicher spaces $f : FM \to FN$. This defines a faithful
  functor $F : \mathrm{Man} \to \mathrm{Froe}$ that is full on the subcategory
  of hausdorff manifolds.
\end{prop}
\begin{proof}
  Let $C_{M}$ be the set satisfying condition (ii) in the definition of a
  fr\"olicher space with respect to $F_{M}$.  Then $C_{M}$ contains all smooth
  functions $\gamma : \mathbf{R} \to M$, and the forward direction of condition
  (i) is satisfied.  If $\alpha : M \to \mathbf{R}$ is a set function such that
  $\alpha \cdot \gamma : \mathbf{R} \to \mathbf{R}$ is smooth for all $\gamma$
  in $C_{M}$, then in particular $\alpha \cdot \gamma : \mathbf{R} \to
  \mathbf{R}$ is smooth for all smooth maps $\gamma : \mathbf{R} \to M$.  It
  follows from Boman's theorem that $\alpha$ is smooth \citep{Boman1967}.

  If $f : M \to N$ is a smooth function then $f^{*} : C^{\infty}(N) \to
  C^{\infty}(M)$ is an $\mathbf{R}$-algebra homomorphism, so $f$ satisfies
  condition (ii) in the definition of a map of fr\"olicher spaces.  Pullback is
  functorial, giving the functor in the statement, and $F$ is full on the
  subcategory of hausdorff manifolds by \cref{prop:Coo_reflect_smoothness}.
\end{proof}

Since $F$ is full and faithful on the subcategory of hausdorff manifolds, we can
suppress it in that case.  But when $M$ is non-hausdorff, the set $C_{M}$ of
curves in $M$ contains functions that aren't smooth.  Any pair $(p, q)$ of
distinct points of $M$ in the closure of the diagonal gives a function $\gamma :
\mathbf{R} \to M$ satisfying $\gamma(t) = p$ for $t \leq 0$ and $\gamma(t) = q$
otherwise.  Then $\alpha \cdot \gamma$ is the constant function at $\alpha(p) =
\alpha(q)$, which is smooth, putting $\gamma$ in $C_{M}$.  But for any cartesian
neighborhood $U$ of $p$ we have $\gamma^{-1}(U) = (-\infty, 0]$, which isn't
open, meaning that $\gamma$ isn't continuous.  It follows that the functor $F :
\mathrm{Man} \to \mathrm{Froe}$ isn't full, and so in general we must
distinguish between the non-hausdorff manifold $M$ and the fr\"olicher space
$FM$.  Nevertheless, when the codomain is a hausdorff manifold $N$,
\cref{prop:Coo_reflect_smoothness} implies that every fr\"olicher map $f : FM \to
N$ is also smooth.

For any fr\"olicher space $X$, the set $F_{X}$ of functionals on $X$ is an
$\mathbf{R}$-algebra when endowed with pointwise addition, multiplication, and
multiplication by scalars.  And pullback along any fr\"olicher map gives an
$\mathbf{R}$-algebra homomorphism by condition (ii) of the definition of
fr\"olicher maps.  This defines a functor
\[
  C^{\infty} : \mathrm{Froe}^{\text{op}} \to \mathrm{Alg}_{\mathbf{R}}
\]
For any manifold $M$ we have $F_{M} = \mathrm{Man}(M, \mathbf{R})$, so this
notation is consistent with the algebra of functionals functor on the category
of manifolds.  And the argument of \cref{cor:refl_colim} shows that $C^{\infty}$
preserves limits, since $C^{\infty}$ is represented by the fr\"olicher space
$\mathbf{R}$.

The underlying set of a fr\"olicher space $X$ naturally admits two topologies:
the \emph{curve topology} is the final topology with respect to $C_{X}$, and the
\emph{functional topology} is the initial topology with respect to $F_{X}$.
Maps of fr\"olicher spaces are continuous with respect to both.  The curve
topology contains the functional topology, but the reverse inclusion doesn't
hold in general.  However, the notion of hausdorff fr\"olicher space is
univocal, and the subcategory of hausdorff fr\"olicher spaces is reflective,
thanks the following results due to Andrew Stacey \citep{nLab}:
\begin{prop}
  \label{prop:hfroe}
  For any fr\"olicher space, the curve topology is hausdorff if and only if the
  functional topology is hausdorff.  We call a fr\"olicher space
  \emph{hausdorff} if the curve and functional topologies are hausdorff.
\end{prop}
\begin{proof}
  The backward direction holds because the curve topology contains the
  functional topology.  For the other direction, suppose that the curve topology
  is hausdorff.  To show that the diagonal of $X$ is closed in the
  functional topology, it suffices to show that it is the intersection of the
  preimages of the diagonal of $\mathbf{R}$ under all functionals.

  So suppose for contradiction that $p$ and $q$ are distinct points such that
  $\alpha(p) = \alpha(q)$ for all functionals $\alpha$ on $X$. Let $\gamma :
  \mathbf{R} \to X$ be the function satisfying $\gamma(t) = p$ for $t \leq 0$
  and $\gamma(t) = q$ otherwise.  Then $\alpha \cdot \gamma$ is the constant
  function at $\alpha(p) = \alpha(q)$ for any $\alpha$, which is smooth, putting
  $\gamma$ in $C_{M}$.  Since the curve topology is hausdorff and $p$ and $q$
  are distinct, we can choose a neighborhood $U$ of $p$ that's open in the curve
  topology and doesn't contain $q$.  But then $\gamma^{-1}(U) = (-\infty, 0]$,
  so $\gamma$ isn't continuous, contradicting the definition of the curve
  topology.
\end{proof}
\begin{prop}
  The full subcategory $\mathrm{HFroe} \hookrightarrow \mathrm{Froe}$ on the
  hausdorff fr\"olicher spaces has a reflection $L : \mathrm{Froe} \to
  \mathrm{HFroe}$.
\end{prop}
\begin{proof}
  Any fr\"olicher space $X$ gives a map of fr\"olicher spaces $X \to
  \prod_{\alpha \in F_{X}} \mathbf{R}$ whose $\alpha$th component is $\alpha$.
  Let $\eta_{X} : X \to LX$ be the coequalizer of the kernel pair of this map.
  The naturality of everything in sight makes $L$ a functor and $\eta_{X}$
  natural in $X$.

  On underlying sets, the map $\eta_{X} : X \to LX$ is the quotient
  such that $\eta_{X}(p) = \eta_{X}(q)$ just in case $\alpha(p) = \alpha(q)$ for
  all functionals $\alpha$ on $X$.
  The set of functionals $F_{LX}$ is such that $\eta_{X}^{*} : F_{LX} \to F_{X}$
  is a bijection.
  So if $\eta_{X}(p)$ and $\eta_{X}(q)$ are distinct points of $LX$,
  there's some functional $\alpha = \alpha^{\dagger} \cdot \eta_{X}$ on $X$ such
  that $\alpha(p)$ and $\alpha(q)$ are distinct, implying that $(\eta_{X}(p),
  \eta_{X}(q))$ doesn't belong to $(\alpha^{\dagger} \times
  \alpha^{\dagger})^{-1}(\Delta_{\mathbf{R}})$.  Since $\eta_{X}$ is surjective,
  this means the functional topology on $LX$ is hausdorff, so $L$ factors
  through the full subcategory of hausdorff fr\"olicher spaces.

  If $X$ is hausdorff, then for any distinct $p$ and $q$ in $X$ the argument of
  \cref{prop:hfroe} gives a functional $\alpha$ on $X$ such that $\alpha(p)$ and
  $\alpha(q)$ are distinct, making $\eta_{X} : X \to LX$ an isomorphism.  So $L$
  is left adjoint to the inclusion $\mathrm{HFroe} \hookrightarrow
  \mathrm{Froe}$.
\end{proof}
\begin{prop}
  \label{prop:froe_top}
  A manifold $M$ is hausdorff if and only if the fr\"olicher space $FM$ is
  hausdorff.
\end{prop}
\begin{proof}
  The backwards direction holds because for any $M$ the three topologies on its
  underlying set are related as
  \[
    \{\text{functional}\}
    \subseteq
    \{\text{curve}\}
    \subseteq
    \{\text{manifold}\}
  \]
  The first inclusion holds for all fr\"olicher spaces.  For the second, upon
  passing to charts it suffices to consider the case $M = \mathbf{R}^{n}$.
  Suppose that $A \subseteq \mathbf{R}^{n}$ isn't open in the manifold topology.
  Then there's some $p$ in $A$ such that for all $n$ we can choose a point
  $p_{n}$ not in $A$ satisfying $\lvert p_{n} - p \rvert < e^{-n}$.  By the
  Special Curve Lemma \citep[\S2.8]{Kriegl1997}, the infinite polygon through
  the $p_{n}$ can be parametrized to give a smooth curve $\gamma : \mathbf{R}
  \to \mathbf{R}^{n}$ satisfying $\gamma(1/n) = p_{n}$ for all $n$ and
  $\gamma(0) = p$.  Then $\gamma^{-1}(A)$ contains $0$ but not $1/n$ for any
  $n$, meaning it's not open and so $A$ is not open in the curve topology.

  Conversely, if $M$ is hausdorff and $U$ is a manifold open, then for any point
  $p$ of $U$ we can use a partition of unity argument to construct a functional
  $\alpha$ on $M$ that vanishes outside of $U$ and satisfies
  $\alpha(p) = 1$.  Then $\alpha^{-1}(\mathbf{R} \setminus \{0\})$ is a
  functional open neighborhood of $p$ contained in $U$.  So $U$ is a functional
  open, and the three topologies coincide.
\end{proof}

Summarizing the situation, we have a pullback of categories
\[
  \begin{tikzcd}
    {\mathrm{HMan}} \ar[r, hook] \ar[d, hook] 
    \ar[dr, phantom, "\lrcorner" very near start] &
    {\mathrm{HFroe}} \ar[d, hook] \\
    {\mathrm{Man}} \ar[r, "F"] &
    {\mathrm{Froe}}
  \end{tikzcd}
\]
with the leg on the right a reflective subcategory inclusion.  The category of
fr\"olicher spaces is complete and cocomplete, making its reflective
subcategories complete and cocomplete as well.  The full and faithful inclusions
reflect limits and colimits, and the relevant facts about colimits are
finished off by the following argument, also due to Andrew Stacey \citep{nLab}:
\begin{prop}
  \label{prop:LF_preserve_colim}
  The functors $\mathrm{HMan} \hookrightarrow \mathrm{HFroe}$ and
  $LF : \mathrm{Man} \to \mathrm{HFroe}$ preserve colimits.
\end{prop}
\begin{proof}
  Let $\theta : M_{\bullet} \to M$ be a colimiting cocone of manifolds or
  hausdorff manifolds, and let $\xi : FM_{\bullet} \to X$ be a colimiting cocone
  in $\mathrm{Froe}$.  The universal property of $\xi$ gives a unique map $f : X
  \to FM$ such that $f \cdot \xi = \theta$.  Since $\theta$ and $\xi$ are both
  colimiting and $C^{\infty}$ preserves limits, the map $f^{*} : C^{\infty}(M)
  \to F_{X}$ is a bijection.  It therefore suffices to show that $Lf$ is a
  bijection.

  For injectivity, suppose that $p$ and $q$ are two points of $X$ such that
  $\eta_{X}(p)$ and $\eta_{X}(q)$ are distinct elements of $LX$.  Then there's
  some functional $\alpha$ on $X$ such that $\alpha(p)$ and $\alpha(q)$ are
  distinct.  Since functionals on $X$ factor uniquely through $f$, this gives a
  functional $\alpha^{\dagger}$ on $M$ such that $\alpha^{\dagger}(f(p))$ and
  $\alpha^{\dagger}(f(q))$ are distinct.  Therefore $f(p)$ and $f(q)$ aren't
  identified in $LFM$, making $Lf$ injective.

  Suppose for contradiction that $f$ isn't surjective, so that there's some $p$
  in $M$ not in the image of $f$.  Let $M^{\dagger} = M \setminus \{p\}$, and
  let $i : M^{\dagger} \to M$ be the inclusion.  The image of $f$ contains the
  image of the cocone $\theta$, so the latter factors as a cocone
  $\theta^{\dagger} : M_{\bullet} \to M^{\dagger}$ such that $\theta = i \cdot
  \theta^{\dagger}$.  By the universal property of the colimiting cocone
  $\theta$, this gives a unique smooth function $r : M \to M^{\dagger}$ such
  that $r \cdot \theta = \theta^{\dagger}$.  But then we have $i \cdot r \cdot
  \theta = \theta$, and since $\theta$ is colimiting it follows that $i \cdot r$
  is the identity of $M$.  This implies that $i$ is surjective, a contradiction.
  So $f$ and therefore $Lf$ are surjective.
\end{proof}

It follows that the hausdorff reflection and the quotient $M/E$ always exist and
coincide as hausdorff fr\"olicher spaces, and so when one is a manifold so is
the other:
\begin{prop}[\cref{thm:quotient}, converse]
  If $M$ is a manifold with a hausdorff reflection, then the quotient $M/E$
  exists.
\end{prop}

\begin{proof}
  Let $\eta : M \to N$ be a hausdorff reflection for $M$.  Choosing a diagram
  $M_{\bullet}$ of hausdorff manifolds and colimiting cocone $\theta :
  M_{\bullet} \to M$ of manifolds, the cocones $\eta \cdot \theta : M_{\bullet}
  \to N$ and $LF\theta : M_{\bullet} \to LFM$ are colimiting cocones of
  hausdorff fr\"olicher spaces by \cref{cor:refl_colim,prop:LF_preserve_colim},
  so $LF\eta : LFM \to N$ is an isomorphism.  Therefore $\eta : M \to N$ is the
  topological quotient of $M$ by $E$.  Since $\eta$ is a local diffeomorphism
  (\cref{prop:local_diffeo}) it's a submersion, making $\eta : M \to N$ the
  quotient of $M$ by $E$ as a manifold.
\end{proof}

\pdfbookmark{References}{References}
\bibliographystyle{alpha}
\bibliography{references}

\begin{thebibliography}{{nLa}23}

\bibitem[Bom67]{Boman1967}
J.~Boman.
\newblock Differentiability of a function and of its compositions with functions of one variable.
\newblock {\em Mathematica Scandinavica}, 20(2):249--268, 1967.

\bibitem[Bou71]{Bourbaki1971}
N.~Bourbaki.
\newblock {\em Vari{\'e}t{\'e}s diff{\'e}rentielles et analytiques}.
\newblock Springer, 1971.

\bibitem[Fr{\"o}80]{Froelicher1980}
A.~Fr{\"o}licher.
\newblock Cat{\'e}gories cart{\'e}siennement ferm{\'e}es engendr{\'e}es par des mono{\"\i}des.
\newblock {\em Cahiers de Topologie et G{\'e}om{\'e}trie Diff{\'e}rentielle Cat{\'e}goriques}, 21(4):367--375, 1980.

\bibitem[Fr{\"o}82]{Froelicher1982}
A.~Fr{\"o}licher.
\newblock Smooth structures.
\newblock In K.~H. Kamps, D.~Pumpl{\"u}n, and W.~Tholen, editors, {\em Category Theory}, pages 69--81. Springer, 1982.

\bibitem[H{\'a}j71a]{Hajicek1971a}
P.~H{\'a}ji{\v c}ek.
\newblock Birfurcate space-times.
\newblock {\em Journal of Mathematical Physics}, 12:157--160, 1971.

\bibitem[H{\'a}j71b]{Hajicek1971b}
P.~H{\'a}ji{\v c}ek.
\newblock Causality in non-{H}ausdorff space-times.
\newblock {\em Communications in Mathematical Physics}, 21:75--84, 1971.

\bibitem[HE73]{Hawking1973}
S.~W. Hawking and G.~F.~R. Ellis.
\newblock {\em The Large Scale Structure of Space-time}.
\newblock Cambridge University Press, 1973.

\bibitem[HR57]{Haefliger1957}
A.~Haefliger and G.~Reeb.
\newblock Vari{\'e}t{\'e}s (non s{\'e}par{\'e}es) {\`a} une dimension et structures feuillet{\'e}es du plan.
\newblock {\em L'Enseignement Math{\'e}matique}, 3:107--125, 1957.

\bibitem[KM97]{Kriegl1997}
A.~Kriegl and P.~W. Michor.
\newblock {\em The Convenient Setting of Global Analysis}.
\newblock American Mathematical Society, 1997.

\bibitem[KMS93]{Kolar1993}
I.~Kol{\'a}{\v r}, P.~W. Michor, and J.~Slov{\'a}k.
\newblock {\em Natural Operations in Differential Geometry}.
\newblock Springer, 1993.

\bibitem[Mil65]{Milnor1965}
J.~W. Milnor.
\newblock {\em Topology from the Differentiable Viewpoint}.
\newblock The University Press of Virginia, 1965.

\bibitem[MS74]{Milnor1974}
J.~W. Milnor and J.~D. Stasheff.
\newblock {\em Characteristic Classes}.
\newblock Princeton University Press, 1974.

\bibitem[{nLa}23]{nLab}
{nLab authors}.
\newblock {{F}}r\"{o}licher space.
\newblock \href{https://ncatlab.org/nlab/show/Fr%C3%B6licher+space}{\texttt{https://ncatlab.org/nlab/show/Fr\%}} \href{https://ncatlab.org/nlab/show/Fr%C3%B6licher+space}{\texttt{C3\%B6licher+space}}, July 2023.
\newblock \href{https://ncatlab.org/nlab/revision/Fr%C3%B6licher+space/64}{Revision 64}.

\bibitem[Sta11]{Stacey2011}
A.~Stacey.
\newblock Comparative smootheology.
\newblock {\em Theory and Applications of Categories}, 25(4):64--117, 2011.

\bibitem[Wed16]{Wedhorn2016}
T.~Wedhorn.
\newblock {\em Manifolds, Sheaves, and Cohomology}.
\newblock Springer, 2016.

\bibitem[Whi36]{Whitney1936}
H.~Whitney.
\newblock Differentiable manifolds.
\newblock {\em Annals of Mathematics}, pages 645--680, 1936.

\bibitem[WW23]{Wu2023}
J.~Wu and J.~O. Weatherall.
\newblock Between a {S}tone and a {H}ausdorff space.
\newblock Forthcoming in \emph{The British Journal for the Philosophy of Science}, 2023.

\end{thebibliography}

\end{document}